\documentclass[a4paper, 11pt]{article}
\usepackage[a4paper, total={6in, 8in}]{geometry}

\usepackage{graphicx}
\usepackage[mathlines]{lineno}
\usepackage{amsmath}
\usepackage{amsfonts}
\usepackage{mathrsfs}
\usepackage{mathtools,amssymb}
\usepackage{amsthm}
\usepackage{hyperref}

\newtheorem{theorem}{Theorem}[section]
\newtheorem{corollary}{Corollary}[theorem]

\begin{document}

\title{On Information (pseudo) Metric \thanks{Supported by Median Technologies.}}

\author{Pierre Baudot\\ Median Technologies, Les Deux Arcs, 1800 Route des Crêtes Bâtiment B, 06560 Valbonne, France\\ pierre.baudot@gmail.com}

\maketitle              

\begin{abstract}
This short note revisit information metric, underlining that it is a pseudo metric on manifolds of observables (random variables), rather than as usual on probability laws. Geodesics are characterized in terms of their boundaries and conditional independence condition. Pythagorean theorem is given, providing in special case potentially interesting natural integer triplets. This metric is computed for illustration on Diabetes dataset using infotopo package.     

\end{abstract}
\section{Introduction}
\begin{sloppypar}
While Fisher and Wasserstein metric have been the subject of a lot of studies along the development of information geometry, information metric, although directly applicable to discrete systems and machine learning, received few attention.	
The information metric, $V(X,Y)=H(X,Y)-I(X;Y)$ (the difference between Joint entropy and mutual information), was discovered by Shannon \cite{Shannon1953}, rediscovered several times, and developed in a normalized form by Rajski \cite{Rajski1961}. Rajski defined the normalized metric:  $d(X,Y)=1- \frac{I(X;Y)}{H(X;Y)}$ \cite{Rajski1961}. It is the central function in the work of Zurek on the thermodynamic cost of computation \cite{Zurek1989}, further developed notably in the context of Kolmogorov complexity \cite{Bennett1998}, and has further been applied for hierarchical clustering and finding category in data by Kraskov and Grassberger \cite{Kraskov2009}. Te sun Han could show that this metric is indeed unique, unraveling that non-negativity  of information imposes triangle inequality \cite{Han1981}. 
Considering the theorem of Hu Kuo Ting establishing the correspondence of informations functions with set theoretical union, intersection and complement on additive functions \cite{Hu1962,Baudot2019}, it becomes obvious that the information metric is an informational and geometrical exact expression of the classical "score or loss" functions used in machine learning such as the "intersection over union", the Dice Index or the Jaccard distance \cite{Jaccard1901} (only the latter is a metric).\\  	
In probability and thermodynamic, it is very common to consider a probability law (a state) as a point on a manifold, the coordinate of which give the extensive variables such as volume, entropy and energy \cite{Callen1960}. Considering information metric impose the introduction a new variant of it, by considering \textbf{manifold of observables or of random variables} (piecewise linear manifold here), that we find very appealing, intuitive and coherent with the introduction of random variable complexes \cite{Abramsky2011,Baudot2015a,Vigneaux2019}. As in the binary random variable case, information functions provides coordinates in the probability simplex and characterize the probability law (up to finite ambiguity, see theorem 3c \cite{Baudot2019}), for the binary case it does not bring much thing new (roughly, just a kind of non-linear coordinate transformation). However, for n-ary variables, with $n>2$, this simplifies probabilistic systems importantly, a simplification which is justified in all cases where the variables are given apriori (a case that covers all data applications or empirical measures). 
Previous works established that Gibbs-Shannon entropy function $H_k$ can be characterized (uniquely up to the multiplicative constant of the logarithm basis) as the first class of cohomology defined on random variables complexes (realized as the poset of partitions of atomic probabilities), endowed with a Hochschild coboundary operator (with a left action of conditioning). Marginalization correspond to localization and allows to construct Topos of information \cite{Baudot2015a,Vigneaux2019} (see also the related results found independently by Baez, Fritz and Leinster \cite{Baez2014,Baez2011}). Surprisingly, this metric appears as a cocycle in the special case considering a symmetric action of conditioning à la Gerstenhaber and Shack \cite{Gerstenhaber1987,Baudot2019a}. Vigneaux could notably underline the correspondence of the formalism with the theory of contextuality developed by Abramsky, that also considers complex of variables \cite{Abramsky2011,Vigneaux2019}. As a result complexes of random variables appear as a key object in those studies, and the study of the information metric presented here shall be considered in the special case of simplicial complex of random variables which geometrical realization are piecewise linear manifolds of observables. Linear and convex combinations of random variables are studied in the context of information homotopy to be submitted \cite{Baudot2021}. Moreover without proof, we expect that the topology induced by this metric to be the poset topology also called Alexandrov topology corresponding to partition poset and as suggested by the work of Bennequin et al. \cite{Bennequin2020}.    	
\end{sloppypar} 

\section{Information pseudo metric} 

\subsection{Functions definition} 
\paragraph{Entropy.} the joint-entropy is defined  by \cite{Shannon1948} for any joint-product of $k$ random variables  $(X_1,..,X_k)$ with $\forall i \in [1,..,k], X_i\leq \Omega$ and for a probability joint-distribution  $\mathbb{P}_{(X_1,..,X_k)}$:
\begin{equation} \label{jointentropy multiple}
H_k = H(X_{1},..,X_{k};P) = k\sum_{x_1 \in [N_1],..,x_k \in  [N_k]}^{N_1\times..\times N_k}p(x_1,..,x_k)\ln p(x_1,..,x_k)\\
\end{equation}
where $[N_1\times...\times N_k]$ denotes the "alphabet" of $(X_1,...,X_k)$. More precisely, $H_k$ depends on 4 arguments: first, the sample space: a finite set $N_{\Omega}$; second a probability law $P$ on $N_{\Omega}$;  third, a set of random variable on $N_{\Omega}$, which is a surjective map $X_{j}:N_{\Omega} \rightarrow N_j$ and provides a partition of $N_{\Omega}$, indexed by the elements $x_{j_i}$ of $N_j$.  $X_j$ is less fine than $\Omega$, and write $X_{j}\leq \Omega$, or $\Omega \rightarrow X_{j}$, and the joint-variable $(X_i, X_j)$ is the less fine partition, which is finer than $ X_i $ and $ X_j $; fourth, the arbitrary constant $k$. Adopting this more exhaustive notation, the entropy of $X_{j}$ for $P$ at $\Omega$ becomes $H_{\Omega}(X_{j};P)=H(X_{j};P_{X_j})=H(X_{j*}(P))$, where $X_{j*}(P)$ is the \emph{marginal} of $P$ by $X_{j}$ in $\Omega$.

\paragraph{Multivariate Mutual informations.} The k-mutual-information (also called co-information) are defined by \cite{McGill1954,Hu1962}:
\begin{equation}\label{n-mutual information}
I_k=I(X_{1};...;X_{k};P) = k\sum_{x_1,...,x_k\in [N_1\times...\times N_k]}^{N_1\times...\times N_k}p(x_1.....x_k)\ln \frac{\prod_{I\subset [k];card(I)=i;i \ \text{odd}} p_I}{\prod_{I\subset [k];card(I)=i;i \ \text{even}} p_I} 
\end{equation}
For example, $I_2=k\sum p(x_1,x_2)\ln \frac{p(x_1)p(x_2)}{p(x_1,x_2)}$ and  the 3-mutual information is the function $I_3=k\sum p(x_1,x_2,x_3)\ln\frac{p(x_1)p(x_2)p(x_3)p(x_1,x_2,x_3)}{p(x_1,x_2)p(x_1,x_3)p(x_2,x_3)}$. 
We have the alternated sums or inclusion-exclusion rules \cite{Hu1962,Matsuda2001,Baudot2015a}:
\begin{equation}\label{Alternated sums of information}
I_n=I(X_1;...;X_n;P)=\sum_{i=1}^{n}(-1)^{i-1}\sum_{I\subset [n];card(I)=i}H_i(X_I;P)
\end{equation}
And the dual inclusion-exclusion relation (\cite{Baudot2019a}): 
\begin{equation}\label{Alternated sums of entropy}
H_n=H(X_1,...,X_n;P)=\sum_{i=1}^{n}(-1)^{i-1}\sum_{I\subset [n];card(I)=i}I_i(X_I;P)
\end{equation}

\paragraph{Conditional Mutual informations.} The conditional mutual information of two variables $X_{1};X_{2}$ knowing $X_3$ is noted $X_3.I(X_{1};X_{2})$ and defined as \cite{Shannon1948}: 
\begin{equation}\label{conditional mutual information}
I(X_{1};X_{2}|X_3;P)=k\sum_{x_1,x_2,x_3\in [N_1\times N_2\times N_3]}^{N_1\times N_2\times N_3} p(x_1,x_2,x_3)\ln \frac{p(x_1,x_3)p(x_2,x_3)}{p(x_3)p(x_1,x_2,x_3)}
\end{equation} 

$H_k$ and $I_k$ allows to obtain  information distance or metric defined by: 
\begin{equation}\label{Information distances}
\begin{split}
V_2 & =V(X,Y;P)= H(X;Y;\mathbb{P})-I(X;Y;\mathbb{P})= H(X|Y;\mathbb{P})+H(Y|X;\mathbb{P}) \\
& = k\sum_{x_1,x_2\in\mathscr{X}}^{N_1*N_2}p(x_1.x_2)\ln \frac{(p(x_1.x_2))^2}{p(x_1)p(x_2)} 
= 2k\mathbb{E}_{X_1.X_2} \ln  \frac{(p(x_1.x_2))}{\sqrt{p(x_1)p(x_2)}} \\
& = D(P_{X_1\times X_2}||P_{X_1}) +  D(P_{X_1\times X_2}||P_{X_2}) 
\end{split}
\end{equation}
The last expression underlines its direct expression as a Jensen-Shannon Divergence. Just as for entropy the multiplicative constant $k$ is arbitrary, the usual convention as $k=-1/\ln 2$ is used here to provide the "Bit" as unit, but  one may see it geometrically as a conformal factor fixing information gauge \cite{Cartan1946}, or other projective metric. Information (pseudo-)metric generalizes to the multivariate case to k-volumes  \cite{Baudot2019a}:
\begin{equation}\label{Information volumes}
V_k = V(X_1,...,X_k;P) = H(X_1;...;X_k;\mathbb{P})-I(X_1,...,X_k;\mathbb{P})\\
\end{equation}
$V_k$ are non-negative and symmetric functions: like $H_k$ and $I_k$ they are invariant to the permutation of the variables, but they have no cohomological interpretation.  \\
Hu Kuo Ting \cite{Hu1962} characterized Markov chains in terms of pairwise mutual information :\\
\begin{theorem} \textbf{(information characterization of Markov chains, Hu Kuo Ting):} \label{information characterization Markov}	
	The variables $X_1,...,X_n$ can be arranged in a Markov process $(X_{i_1},...,X_{i_n})$ if and only if, for every subset $J=\{j_1,...,j_{k-2}\}$ of $\{i_2,...,i_{n-1}\}$ of cardinality $k-2$, we have $I_k(X_{i_1};X_{j_1},...;X_{j_{k-2}};X_{i_n})=I_2(X_{i_1};X_{i_n}).$
\end{theorem} 

\begin{figure} [!h]  	
	\centering
	\includegraphics[height=2.5cm]{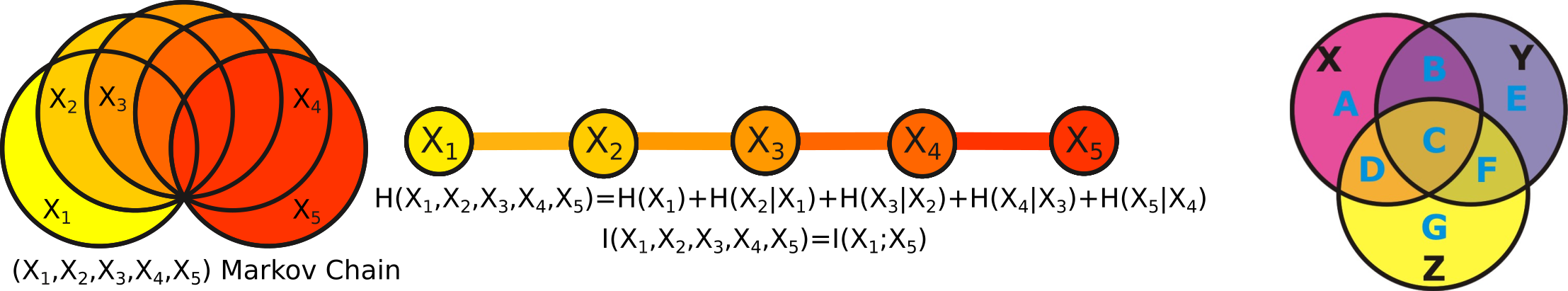}
	\caption{(Left) Markov chains: corresponding Venn Diagram and undirected graph.  (Right) Venn diagram corresponding to the information decomposition for the proof of Information triangle inequality}
	\label{figure Markov}
\end{figure}

\noindent As a consequence, all the functions $I_k(X_I)$ involving $i_1$ and $i_n$  are positive for a Markov process between $(X_{i_1},...,X_{i_n})$. Equivalently, we have $X_1\rightarrow...\rightarrow X_n$  forms a Markov chain if and only if $\forall I \subseteq [n]/{1,n}$ we have $X_I.I(X_1;X_n)=0$. As a special case we have $X_1 \rightarrow X_2 \rightarrow X_3$ forms a Markov chain if and only if  $X_2.I(X_1,X_3)=0$ (cf. Figure \ref{figure Markov} left)

\subsection{Information Pseudo Metric} 

\begin{sloppypar}
Information metric is a pseudometric rather than a metric, since we can find cases for which $H(X;Y;\mathbb{P})-I(Y;X;\mathbb{P})=0$ but clearly $X\neq Y$, indeed all the points (probability laws) satisfying the equation: 
$\prod_{x,y}p(x,y)^{2p(x,y)}=\prod_{x,y}p(x)^{p(x)}p(y)^{p(y)}$. For example, considering two binary random variables, the preceding equation becomes:  
\begin{equation}
\scriptstyle P_{00}^{2P_{00}}P_{01}^{2P_{01}}P_{10}^{2P_{10}}P_{11}^{2P_{11}}=(P_{00}+P_{01})^{P_{00}+P_{01}}(P_{00}+P_{10})^{P_{00}+P_{10}}(P_{11}+P_{01})^{P_{11}+P_{01}}(P_{11}+P_{10})^{P_{11}+P_{10}}
\end{equation}

\noindent Let's note $\{P_{00},P_{01},P_{10},P_{11}\}$ the probability coordinates, then $\{1/2,0,0,1/2\}$, $\{0,1/2,1/2,0\}$, $\{1,0,0,0\}$, $\{0,1,0,0\}$, $\{0,0,1,0\}$ and $\{0,0,0,1\}$ are solutions of $V(X,Y)=0$, e.g. all maxima of $I_2$ and the fully deterministic cases. For $\{0,1,0,0\}$, $\{0,0,1,0\}$, the marginal probability laws are not equal, just equivalent under the permutation of the atoms (e.g. for $\{0,1,0,0\}$, we have for the first variable $P(X_1=0)=1,P(X_1=1)=0$ while for the second variable we have $P(X_2=0)=0,P(X_2=1)=1$). This probably generalizes to arbitrary discrete probability law. 
If we identify the sets of probability laws with $0$ pseudometric (characterized below) into a single equivalence class and quotient the information structure by this equivalence class, then the resulting quotient information structure can be properly metrized by the induced metric. The sets with 0 pseudometric are the elementary events of the marginal random variable, the equivalence class can be identified with the barycentric center of those atoms on the probability simplex, the maximum entropy point of $(X;\mathbb{P})$.
\end{sloppypar}
\begin{theorem}[information pseudo metric]
	$V_2$ is a pseudo metric: It fulfills the 3 axioms of pseudometric, namely:
	\begin{itemize}
		\item \textbf{symmetry:} $V(X,Y;\mathbb{P})=V(Y,X;\mathbb{P})$
		\item \textbf{For metric: identity of indiscernible :} $(V(X,Y;\mathbb{P})=0)\Leftrightarrow X=Y$.\\
		\textbf{For pseudometric: equivalence of indiscernible}  $V(X,X;\mathbb{P})=0$ (a weakening of the previous).
		\item \textbf{triangle inequality:} $V(X,Z;\mathbb{P})\leq V(X,Y;\mathbb{P})+V(Y,Z;\mathbb{P})$ 
	\end{itemize} 
\end{theorem}
Positivity $H(X,Y;\mathbb{P})\geq 0$ follows from the 3 axioms. Literally, a pseudometric space generalizes metric space in the sense that points need not be distinguishable like in metric space:  formally, one may have $d(X,Y) = 0$ for distinct points $X\neq Y$.\\
\begin{proof}
	The proof of the first criterion just follows from the commutativity of addition in information. The proof of the second axiom $V(X,X;\mathbb{P})=0$ follows from the fact that $H(X|X)=0$. The proof of the triangle inequality is provided by considering the information decomposition as for example depicted in the Figure \ref{figure Markov} right. For simplicity, using set theoretic notations of Entropy and Mutual Information, and we consider that $H(X|Y\cup Z)=A$, $H(X\cap Y|Z)=B$, $H(X\cap Y\cap Y)=C$, $H(X\cap Z|Y)=D$, $H(Y|X\cup Z)=E$, $H(Y\cap Z|X)=F$, $H(Z|X\cup Y)=G$. Then, for whatever random variable $X,Y,Z$, the previous triangle inequality can be written $A+B+G+F\leq A+B+G+F+2E+2D$, which gives $0\leq2E+2D$ or $0\leq2H(Y/X\cup Z)+2H(X\cap Z/Y)$ which by non negativity of the conditional and pairwise Mutual Information is always true. This holds only in the case where the logarithm basis $c$ is chosen in $]0,1]$.
\end{proof}

\subsection{Information geodesics} 

The cases for which the triangle inequality is an equality is interesting since it accounts for the  basic notion of "straight line" or "shortest path". As illustrated in Figure \ref{figure_Information triangleInequality and Markov} any 3 variables $X,Y,Z$ define 3 different triangle inequalities and sub cases of equality. We note those 3 triangle equality $(X,Y,Z)$, $(Y,X,Z)$, $(X,Z,Y)$. We call those cases for which the triangle equality holds, geodesic $(X,Y,Z)$ or geodesic $(Y,X,Z)$, or geodesic $(X,Z,Y)$, although it will only get some more precise meaning after the introduction of complexes of random variable, allowing to define piece-wise linear manifolds, and piecewise linear geodesics of random-variables. We have the following theorem:\\
\begin{theorem} \textbf{ A Geodesic $(X,Y,Z)$ is a Markov chain only determined by its boundaries $X$ and $Z$ :} \label{Information geodesics and Markovianity}  
	A totally ordered triplet $(X,Y,Z)$ is geodesic if and only if $(X;Z)$ are conditionally independent given $Y$ and $H(Y|(X,Z))=0$.
\end{theorem}
\begin{corollary}
	if $(X,Y,Z)$ is geodesic then  $(X,Y,Z)$ form a Markov chain and $I(X,Y,Z)$ is non negative.
\end{corollary}		
\begin{proof}
	A totally ordered triplet $(X,Y,Z)$ is geodesic if and only if $V(X,Z)=V(X,Y)+V(Y,Z)$ which is $H(X;Z)-I(X;Z)=H(X;Y)-I(Y;X)+H(Y;Z)-I(Y;Z)$. The equality $H(X;Z)-I(X;Z)=H(X;Y)-I(Y;X)+H(Y;Z)-I(Y;Z)$ holds if and only if $H(Y|(X,Z))+I(X;Z|Y)=0$. Since both terms in the left part of the equation are nonnegative and independent \cite{Han1975}, a necessary and sufficient condition is that both vanish $H(Y|(X,Z))=0$ and $I(X;Z|Y)=0$. $H(Y|(X,Z))=0$  is equivalent to $Y\subset X\cup Z$, meaning that $Y$ is fully determined by $X\cup Z$,  and $I(X;Z|Y)=0$ is equivalent to the requirement that $(X;Z)$ are conditionally independent given $Y$ and hence to the requirement that $(X,Y,Z)$ form a Markov chain (see \ref{information characterization Markov}).
\end{proof} 

\begin{figure} [!h]  	
	\centering
	\includegraphics[height=1.7cm]{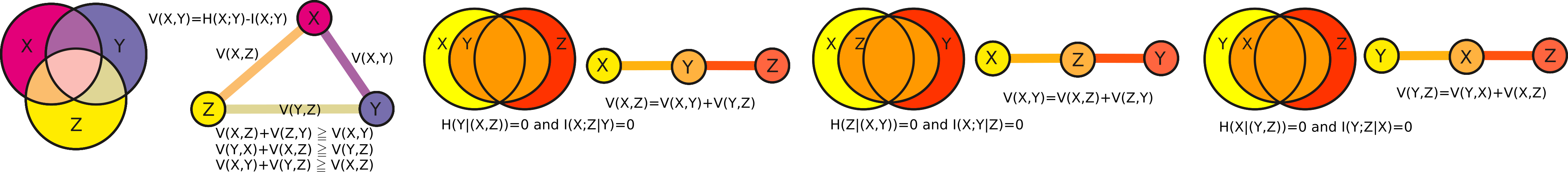}
	\caption{The 3 information triangle inequalities (left) and the associated 3 special case equalities with the corresponding Markov chains (right), together with their associated Venn diagrams.}
	\label{figure_Information triangleInequality and Markov}
\end{figure}

More roughly, it shows that if the path between $(X,Y,Z)$ is of "minimum length, or aligned", then $(X,Y,Z)$ form a Markov chain and $I_3$ is positive. The 3 triangle inequality, and the 3 cases of equality associated with their Markov Chains are depicted in the figure \ref{figure_Information triangleInequality and Markov} by their corresponding undirected graph and Venn diagrams. As the constraint $H(Y|(X,Z))=0$ only imposes the inclusion of $Y$ to the geodesic $(X,Y,Z)$, we see that the constraint of conditional independence $I(X;Z|Y)=0$ imposes the "straightness", hence one may interpret geometrically conditional dependences  $I(X;Z|Y)$ as quantifying the deviation from straight line.

We call a totally ordered k-uplet $(X_1,...,X_k)$ a conditionally independent chain if for all sub total orders of 3 variables $(X_h,X_i,X_k)$ we have $I(X_h;X_j|X_i)=0$. We call a totally ordered k-uplet $(X_1,...,X_k)$  a deterministic chain if for all sub total orders of 3 variables $(X_h,X_i,X_k)$ we have $H(X_i|X_h,X_j)=0$ (which is equivalent to claim that  to $X_i\subset X_h\cup X_j$, meaning that $X_i$ is  deterministic function of $X_h\cup X_j$ ). It directly generalizes to arbitrary $k$ random variables:
\begin{theorem}[general random variable geodesics]
	\textbf{ } \label{Information geodesics general}  
	A totally ordered k-uplet $(X_1,...,X_k)$ is a Geodesic if and only if it is a conditionally independent and deterministic chain (determined by its boundaries $X_1$ and $X_k$).
\end{theorem}
\begin{corollary}
	if  $(X_1,...,X_k)$ is a geodesic then  $(X_1,...,X_k)$ form a Markov chain and all $I_k$ are non negative.
\end{corollary}	

\begin{proof}
	It is trivial from the definition and the preceding theorem. The requirement that $(X_1,...,X_k)$ is a geodesic is equivalent to require that all the  $\binom{3}{k}$ triplets in $k$ are geodesic, and hence to the fact that all the $\binom{3}{k}$ both  conditional independence $I(X_h;X_j|X_i)=0$ with $h<i<j$, and conditional entropies $H(X_i|X_h;X_j)=0$ with $h<i<j$, holds. 
\end{proof}

\begin{figure} [!h]  	
	\centering
	\includegraphics[height=2.5cm]{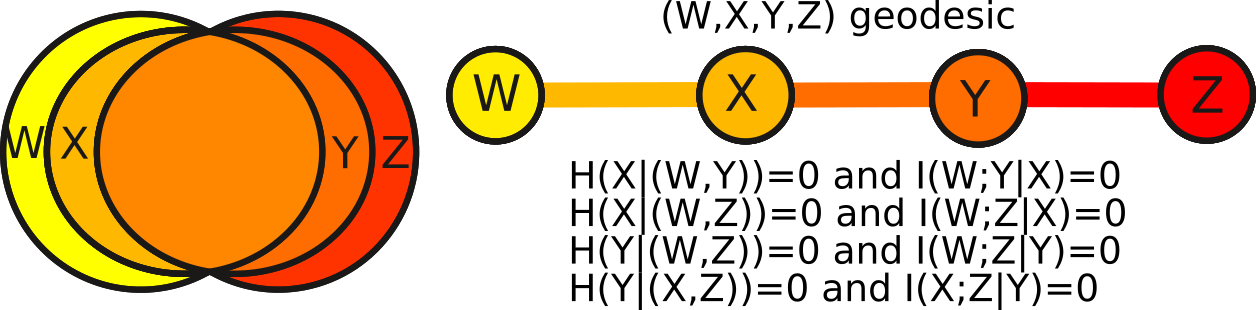}
	\caption{A 4-geodesic with its associated conditional information and entropy, together with its associated Venn diagrams.}
	\label{figure_Information triangleInequality and Markov}
\end{figure}

\subsection{Pythagorean Theorem for Information Metric} 

The second case of interest is the one that fulfills Pythagoras theorem, that characterize orthogonality in Euclidean geometry. 
Consider a triplet $(X,Y,Z)$ of random variable, we call the triplet $(X,Y,Z)$ a Pythagorean triplet if it satisfies Pythagoras relation (cf. Figure \ref{pythagoras}), then we have: 
\begin{theorem} [Pythagorean theorem of information]\label{Pythagorean theorem of information}  
	A triplet $(X,Y,Z)$ is Pythagorean if and only if it satisfies  one of the  3  equations obtained by cyclic permutation of $(X,Y,Z)$ on the following equation:
	\begin{equation}\label{Pythagorean equation}
	\begin{split}
	&\scriptstyle  \sum_{I\subset[N_X \times N_Y], |I|=2} P_I^2 \left( \log \frac{P_I^2}{P_{I_1}P_{I_2}} \right)^2  + 2k \sum_{I\subset[N_X \times N_Y], |I|=2, J\subset[N_X \times N_Y], |I|=2, I<J} P_I P_J  \log \frac{P_I^2}{P_{I_1}P_{I_2}} \log \frac{P_J^2}{P_{J_1}P_{J_2}}  \\
	&\scriptstyle =  \sum_{I\subset[N_X \times N_Z], |I|=2} P_I^2 \left( \log \frac{P_I^2}{P_{I_1}P_{I_2}} \right)^2  + 2k \sum_{I\subset[N_X \times N_Z], |I|=2, J\subset[N_X \times N_Z], |I|=2, I<J} P_I P_J  \log \frac{P_I^2}{P_{I_1}P_{I_2}} \log \frac{P_J^2}{P_{J_1}P_{J_2}}  \\
	&\scriptstyle +  \sum_{I\subset[N_Y \times N_Z], |I|=2} P_I^2 \left( \log \frac{P_I^2}{P_{I_1}P_{I_2}} \right)^2  + 2k \sum_{I\subset[N_Y \times N_Z], |I|=2, J\subset[N_Y \times N_Z], |I|=2, I<J} P_I P_J  \log \frac{P_I^2}{P_{I_1}P_{I_2}} \log \frac{P_J^2}{P_{J_1}P_{J_2}}  
	\end{split}	
	\end{equation}
	where $P_{J_1}$ and $P_{J_1}$ denotes the two marginal variables of the pair $J$.
\end{theorem}
\begin{figure} [!h]  	
	\centering
	\includegraphics[height=2.2cm]{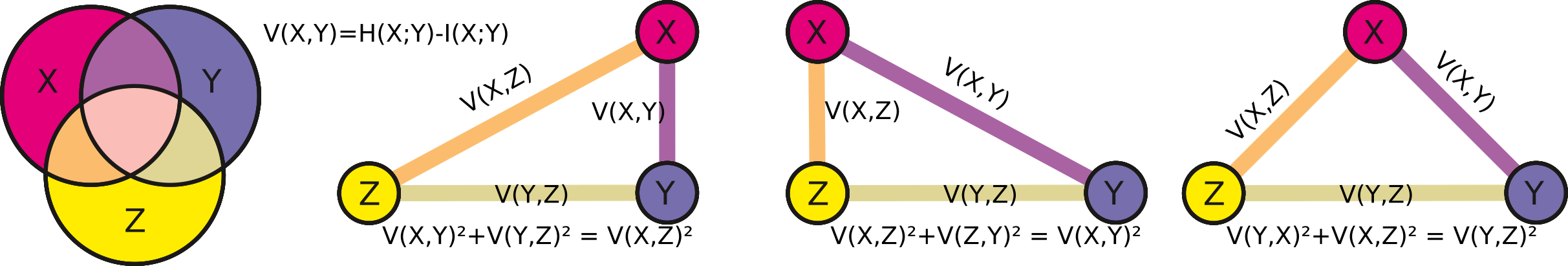}
	\caption{The 3 cases of Pythagorean equality of squared distance corresponding to Euclidean orthogonality and their associated  equations.}
	\label{pythagoras}
\end{figure}

\begin{proof}
	There are 3 possible Pythagorean equation obtained by cyclic permutation of $(X,Y,Z)$. One Pythagorean equation for information distance is $V(X;Y;P_{X\times Y})^2 =V(X;Z;P_{X\times Z})^2 + V(Y;Z;P_{Y\times Z})^2$. Substituting each distance by its expression, for example the = $V(X;Y;P_{X\times Y}) =k\sum_{x_1,x_2\in\mathscr{X}}^{N_1*N_2}p(x_1.x_2)\ln \frac{(p(x_1.x_2))^2}{p(x_1)p(x_2)}$, and then applying remarkable identity of squared polynomial $(a+b+c)^2=a^2+b^2+c^2+2(ab+ac+bc)$ gives the expected result.
\end{proof}
The expression is slightly cumbersome, considering special cases like identically distributed independent variables simplifies a lot the expression and we obtain, as a corollary, the 3 equations obtained by cyclic permutation of $(X,Y,Z)$, are given by the corollary :
\begin{corollary}
	A triplet $(X,Y,Z)$ is a Pythagorean triplet of independent and identically distributed variables if and only if it satisfies one of the 3 equations obtained by cyclic permutation of $(X,Y,Z)$ on the following equations: $N_Z=c^k$, $N_Y=N_X^{(k)^2}$, where $ k\in \mathbb{N}^+$ and $c\in \mathbb{N},~ c>1$ is the basis of the logarithm.
\end{corollary}	
\begin{proof}
	The variables are independent if and only if  $V_2=H_2$ \cite{Baudot2019}, the variables are independent identically distributed if and only if $H(X,Y)=k\log (N_X.N_Y)$. Hence one of the 3  Pythogorean equation $V(X;Y;P_{X\times Y})^2 =V(X;Z;P_{X\times Z})^2 + V(Y;Z;P_{Y\times Z})^2$ becomes by application of remarkable identity:
	\begin{equation}\label{Pythagorean equation iid}
	\log_c N_X \log N_Y -\log_c N_X \log_c N_Z -\log_c N_Y \log_c N_Z = \left(\log_c N_Z \right)^2
	\end{equation}
	which is equivalent to: $\log_c N_Y^{\log_c N_X} -\log_c N_Y^{\log N_Z} -\log_c N_Z^{\log_c N_X} - \log_c N_Z^{\log_c N_Z} =0 $. A simple algebraic calculus gives $ \log_c \frac{N_Y^{\frac{ \log_c N_X} {\log_c N_Z }}}{N_Z^{\log_c N_X \log_c N_Z}}=0$ and hence $ N_Y= N_X^{(\log_c N_Z)^2}$. By definition the $N_x,N_y,N_z \in  \mathbb{N}^+$ are natural integers in the basic discrete setting, hence the equation holds if and only if $(\log_c N_Z)^2 \in \mathbb{N}^+$, which can only be achieved if  $N_z= c^k$ with $k\in \mathbb{N}^+$, and $c\in \mathbb{N},~ c>1$, because of the transcendence of the logarithm function. Then if $N_z= c^k$  we have $N_Y=N_X^{(k)^2}$, which is the expected result.
\end{proof}
This result suggests  extensions and generalizations to continuous variable and spheric or hyperbolic geometry that are left for further work. It more over provide an unexpected notion of orthogonality in natural integers \cite{Niven1964,Kuipers1971}, the special case of identically distributed but not necessarily independent should be of interest.  

\subsection{Informational metric measure space and optimal transport} 
Defining the metric $V(X ;Y,\mathbb{P})$ turns the information structure $\mathcal{S}$ into a metric space (more exactly a pseudometric space), and since entropy is a measure \cite{Yeung2007}, it can be considered as a (pseudo-)metric measured space. Requiring a function to be a metric and additive is indeed a standard construction of measure, see \cite{Rudin1976} p. 305 (notably for the proof that symmetric difference properties implies triangle inequality). This is always a complete metric space. If it is separable, the measure algebra information structure is also called separable, and indeed any (countably) finitely generated information structure is separable. This metric is known to be invariant under volume-preserving affinities of $\mathbb{R}^n$ \cite{Shephard1965}. This way, it becomes possible to obtain a metric measure space where metric and measure are basic and intrinsically pertain to information theory, such that it becomes possible to investigate optimal transport theory based on Kantorovich-Wasserstein distance on the same footing \cite{Ollivier2009,Figalli2011}. On such a line pointing out that information theory is more general than optimal transport theory (at first, it does not require metric assumption), Belavkin \cite{Belavkin2016} showed that relaxing the constraint of output measure in optimal transport, the optimal transport problem becomes mathematically equivalent to the optimal channel problem in information theory, which uses a constraint on the mutual information and hence that the  optimal channel defines a lower bound on the Wasserstein metric.

\section{Application to data} 
\begin{sloppypar}
The python package \href{https://infotopo.readthedocs.io/en/latest/basic_methods.html#information-distance}{infotopo} computes information distances and volumes within a given datasets \cite{Baudot2020}. It also provide the resulting distances as the adjacency matrix and its associated graph representation. This matrix is a standard input for many machine learning package for clustering like HDBSCAN or dimension reduction like UMAP \cite{McInnes2018}. The package includes examples of applications to several challenge data set provided by \href{https://scikit-learn.org/stable/datasets/index.html}{scikit-learn} \cite{Pedregosa2011}. Methods to estimate the curse of dimensionality (undersampling) and statistical test of independence, as well as information landscapes are described in \cite{Baudot2019}. We provide here the example of the Diabetes dataset, illustrated in figure \ref{Diabete_infometric}. This dataset contains 10 variables-dimensions for a sample size (number of points) of 442 and a target (label) variable which quantifies diabetes progress. The ten variables are [0:age, 1:sex, 2:body mass index, 3:average blood pressure, 4:T-Cells, 5:low-density lipoproteins, 6:high-density lipoproteins, 7:thyroid stimulating hormone, 8:lamotrigine, 9:blood sugar level] in this order. The package allows to compute most of the usual information functions, as presented in \cite{Baudot2019,Tapia2018}.  Higher statistical structure quantified by multivariate Mutual-Informations and Total correlations obviously provide much more discriminative information for supervised and unsupervised learning \cite{Baudot2019,Baudot2019a} (obviously better than multivariate $V_k$ that are not boundary or cocycle). It is possible to identify geodesics of variables, even if they a priori seem unlikely given the hard constraint of deterministic chains, and we let it for further investigations.
\end{sloppypar}
\begin{figure} [!h]  	
	\centering
		\includegraphics[height=11cm]{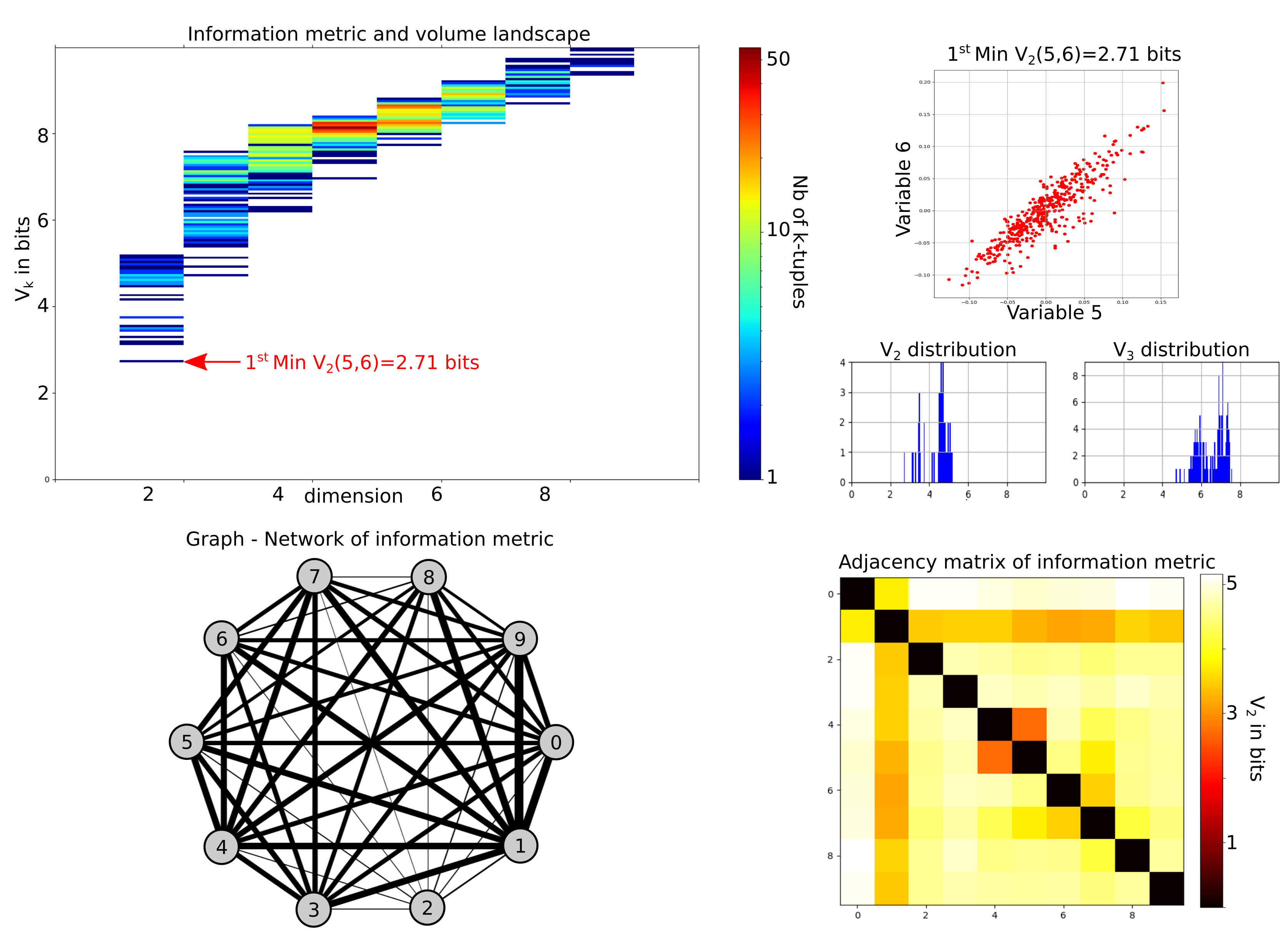}
		\caption{Information metric on Diabete dataset (scikit-learn \cite{Pedregosa2011}): The metric ($V_2$) and Volume landscape (see \cite{Baudot2019,Baudot2019a}). In red, the pair of varaible (5:low-density lipoproteins, 6:high-density lipoproteins) presenting the lowest information metric. In blue: the distribution of $V_2$ and $V_k$ presented in the landscape. Bottom right: the adjacency matrix of the information metric. Bottom left: the associated simple undirected graph of information metric, the thickness of the edges is proportional to the distance.}
		\label{Diabete_infometric}
\end{figure}

\bibliographystyle{splncs04}
\bibliography{bibtopo}

\end{document}